\documentclass{amsart}
\usepackage{hyperref}
\newtheorem{theorem}{Theorem}[section]
\newtheorem{lemma}[theorem]{Lemma}
\newcommand{\C}{\mathbb{C}}

\newcommand{\N}{\mathbb{N}}
\newcommand{\R}{\mathbb{R}}
\newcommand{\T}{\mathbb{T}}
\newcommand{\Z}{\mathbb{Z}}

\newcommand{\cA}{\mathcal{A}}
\newcommand{\cB}{\mathcal{B}}
\newcommand{\cP}{\mathcal{P}}

\newcommand{\im}{\mathrm{Im}\,}
\newcommand{\tr}{\mathrm{tr}}
\newcommand{\eps}{\varepsilon}

\newcommand{\cC}{\mathcal{C}}
\newcommand{\cH}{\mathcal{H}}
\begin{document}
\title[Asymptotic Formulas for Toeplitz Matrices]
{Higher Order Asymptotic Formulas\\
for Toeplitz Matrices with Symbols\\
in Generalized H\"older Spaces}

\author{Alexei~Yu.~Karlovich}
\thanks{The author is supported by F.C.T. (Portugal) grants
SFRH/BPD/11619/2002 and FCT/ FEDER/POCTI/MAT/59972/2004}

\address{%
Departamento de Matem\'atica,
Instituto Superior T\'ecnico,
Av. Rovisco Pais 1,
1049--001, Lisbon,
Portugal}
\email{akarlov@math.ist.utl.pt}
\begin{abstract}
We prove higher order asymptotic formulas for determinants and traces of finite
block Toeplitz matrices generated by matrix functions belonging to generalized
H\"older spaces with characteristic functions from the Bari-Stechkin class.
We follow the approach of B\"ottcher and Silbermann and generalize their results
for symbols in standard H\"older spaces.
\end{abstract}
\maketitle
\section{Introduction}
\subsection{Finite block Toeplitz matrices}
Let $\Z,\N,\Z_+$, and $\C$ be the sets of integers, positive integers,
nonnegative integers, and all complex numbers, respectively.
Suppose $N\in\N$. For a Banach space $X$,
let $X_N$ and $X_{N\times N}$
be the spaces of vectors and matrices with entries in $X$.
Let $\T$ be the unit circle. For $1\le p\le\infty$, let
$L^p:=L^p(\T)$ and $H^p:=H^p(\T)$ be the standard Lebesgue and Hardy
spaces of the unit circle. For $a\in L_{N\times N}^1$ one can define
\[
a_k=\frac{1}{2\pi}\int_0^{2\pi}a(e^{i\theta})e^{-ik\theta}d\theta
\quad (k\in\Z),
\]
the sequence of the Fourier coefficients of $a$.
Let $I$ be the identity operator, $P$ be the Riesz projection of
$L^2$ onto $H^2$, $Q:=I-P$, and define $I,P$, and $Q$ on
$L_N^2$ elementwise. For $a\in L_{N\times N}^\infty$ and $t\in\T$,
put $\widetilde{a}(t):=a(1/t)$ and $(Ja)(t):=t^{-1}\widetilde{a}(t)$.
Define \textit{Toeplitz operators}
\[
T(a):=PaP|\im P,
\quad
T(\widetilde{a}):=JQaQJ|\im P
\]
and \textit{Hankel operators}
\[
H(a):=PaQJ|\im P,
\quad
H(\widetilde{a}):=JQaP|\im P.
\]
The function $a$ is called the \textit{symbol} of $T(a)$, $T(\widetilde{a})$,
$H(a)$, $H(\widetilde{a})$. We are interested in the asymptotic behavior of
\textit{finite block Toeplitz matrices}
\[
T_n(a):=(a_{j-k})_{j,k=0}^n
\]
generated by (the Fourier coefficients of) the symbol $a$ as $n\to\infty$.
Many results about asymptotic properties of $T_n(a)$ as $n\to\infty$ are contained
in the books by Grenander and Szeg\H{o} \cite{GS58}, B\"ottcher and Silbermann
\cite{BS83,BS99,BS06}, Hagen, Roch, and Silbermann \cite{HRS01}, Simon \cite{Simon05},
and B\"ottcher and Grudsky \cite{BG05}.
\subsection{Szeg\H{o}-Widom limit theorems}
Let us formulate precisely the most relevant results.
Let $(K_{2,2}^{1/2,1/2})_{N\times N}$ be the Krein algebra \cite{Krein66}
of matrix functions $a$ in $L_{N\times N}^\infty$
satisfying $\sum_{k=-\infty}^\infty \|a_k\|^2|k|<\infty$,
where $\|\cdot\|$ is any matrix norm on $\C_{N\times N}$.
The following beautiful theorem about the asymptotics of finite
block Toeplitz matrices was proved by Widom \cite{Widom76}.
\begin{theorem}\label{th:Widom1}
{\rm (see \cite[Theorem~6.1]{Widom76})}.
If $a\in (K_{2,2}^{1/2,1/2})_{N\times N}$ and the Toeplitz operators
$T(a)$ and $T(\widetilde{a})$ are invertible on $H_N^2$, then
$T(a)T(a^{-1})-I$ is of trace class and, with appropriate branches
of the logarithm,
\begin{equation}\label{eq:Widom1}
\log\det T_n(a)=(n+1)\log G(a)+\log \mathrm{det}_1 T(a)T(a^{-1})+o(1)
\quad\mbox{as}\quad n\to\infty,
\end{equation}
where $\det_1$ is defined in Section~{\rm\ref{sec:Schatten-Neumann}} and
\[
G(a):=\lim_{r\to 1-0}\exp\left(\frac{1}{2\pi}\int_0^{2\pi}
\log\det\widehat{a}_r(e^{i\theta})d\theta\right),
\quad
\widehat{a}_r(e^{i\theta}):=\sum_{n=-\infty}^{\infty}a_nr^{|n|}e^{in\theta}.
\]
\end{theorem}
The proof of the above result in a more general form is contained in
\cite[Theorem~6.11]{BS83} and \cite[Theorem~10.30]{BS06}.

Let $\lambda_1^{(n)},\dots,\lambda_{(n+1)N}^{(n)}$ denote the eigenvalues of
$T_n(a)$ repeated according to their algebraic multiplicity. Let $\mathrm{sp}\,A$
denote the spectrum of a bounded linear operator $A$ and $\tr M$ denote the trace
of a matrix $M$. Theorem~\ref{th:Widom1} is equivalent to the assertion
\[
\sum_i\log\lambda_i^{(n)}=\tr\log T_n(a)=(n+1)\log G(a)+
\log\mathrm{det}_1 T(a)T(a^{-1})+o(1).
\]
Widom \cite{Widom76} noticed that Theorem~\ref{th:Widom1} yields even a description of
the asymptotic behavior of $\tr f(T_n(a))$ if one replaces $f(\lambda)=\log\lambda$
by an arbitrary function $f$ analytic in an open neighborhood of the union
$\mathrm{sp}\,T(a)\cup\mathrm{sp}\,T(\widetilde{a})$ (we henceforth call such
$f$ simply analytic on $\mathrm{sp}\,T(a)\cup\mathrm{sp}\,T(\widetilde{a})$).
\begin{theorem}\label{th:Widom2}
{\rm (see \cite[Theorem~6.2]{Widom76}).}
If $a\in (K_{2,2}^{1/2,1/2})_{N\times N}$ and if $f$ is analytic on
$\mathrm{sp}\,T(a)\cup\mathrm{sp}\,T(\widetilde{a})$, then
\begin{equation}\label{eq:Widom2}
\tr f(T_n(a))=(n+1)G_f(a)+E_f(a)+o(1)
\quad\mbox{as}\quad n\to\infty,
\end{equation}
where
\begin{eqnarray*}
G_f(a)
&:=&
\frac{1}{2\pi}\int_0^{2\pi}(\tr f(a))(e^{i\theta})d\theta,
\\
E_f(a)
&:=&
\frac{1}{2\pi i}\int_{\partial\Omega}f(\lambda)
\frac{d}{d\lambda}\log\mathrm{det}_1 T[a-\lambda]T[(a-\lambda)^{-1}]d\lambda,
\end{eqnarray*}
$\det_1$ is defined in Section~{\rm\ref{sec:Schatten-Neumann}},
and $\Omega$ is any bounded open set containing the set
$\mathrm{sp}\,T(a)\cup\mathrm{sp}\,T(\widetilde{a})$
on the closure of which $f$ is analytic.
\end{theorem}
The proof of Theorem~\ref{th:Widom2} for continuous symbols $a$ is also given
in \cite[Section 10.90]{BS06} and in \cite[Theorem~5.6]{BS99}.
In the scalar case ($N=1$) Theorems~\ref{th:Widom1} and \ref{th:Widom2}
go back to Gabor Szeg\H{o} (see \cite{GS58} and historical remarks
in \cite{BS83,BS99,BS06,HRS01,Simon05}).
\subsection{Smoothness effects}
Fisher and Hartwig \cite{FH69} were probably the first to draw due attention to
higher order correction terms in asymptotic formulas for Toeplitz determinants.
B\"ottcher and Silbermann \cite{BS80} obtained analogs of Theorem~\ref{th:Widom1}
for symbols belonging to H\"older-Zygmund spaces $C_{N\times N}^\gamma$,
$0<\gamma<\infty$. If $\gamma>1/2$, then $C_{N\times N}^\gamma$ is properly
contained in $(K_{2,2}^{1/2,1/2})_{N\times N}$,
and for $a\in C_{N\times N}^\gamma$, formula \eqref{eq:Widom1} is then valid with
$o(1)$ replaced by $O(n^{1-2\gamma})$. Nowadays this result can be proved almost
immediately by using the so-called Geronimo-Case-Borodin-Okounkov formula (see
\cite{BW06}). The author \cite{K-IWOTA05} proved that if $\gamma>1/2$ and
$a\in C_{N\times N}^\gamma$, then \eqref{eq:Widom2} holds with $o(1)$ replaced
by $O(n^{1-2\gamma})$. That is, for very smooth symbols the remainders in
\eqref{eq:Widom1} and \eqref{eq:Widom2} go to zero with high speed (depending
on the smoothness).

On the other hand, B\"ottcher and Silbermann \cite{BS80} (see also
\cite[Sections~6.15--6.20]{BS83} and \cite[Sections~10.34--10.38]{BS06})
observed that if $0<\gamma\le 1/2$, then \eqref{eq:Widom1} requires a correction
involving additional terms and regularized operator determinants.
This is the effect of ``insufficient smoothness".
They also studied the same problems for Wiener algebras with power weights
\cite{BS80}, \cite[Sections~6.15--6.20]{BS83}, \cite[Sections~10.34--10.38]{BS06}.
Recently the author \cite{K-JMAA} extended their higher order versions of
Theorem~\ref{th:Widom1} to Wiener algebras with general weights satisfying
natural submultiplicativity, monotonicity, and regularity conditions.
Corresponding higher order asymptotic trace formulas are proved in \cite{K-ZAA}
(see also \cite[Section~10.91]{BS06}).

Very recently, it was observed in \cite{BKS-Krein100} that the approach of
\cite{BS80} with some improvements of \cite{K-JMAA} is powerful enough to deliver
higher order asymptotic formulas for Toeplitz determinants with symbols in
generalized Krein algebras $(K_{p,q}^{\alpha,\beta})_{N\times N}$ with
$1<p,q<\infty$, $0<\alpha,\beta<1$, and $1/p+1/q=\alpha+\beta\in(0,1)$.
This is another example of ``insufficient smoothness" because one cannot
guarantee that $T(a)T(a^{-1})-I$ is of trace class whenever
$a\in (K_{p,q}^{\alpha,\beta})_{N\times N}$. Notice that generalized
Krein algebras contain discontinuous functions in contrast to
H\"older-Zygmund spaces and weighted Wiener algebras, which consist of
continuous functions only.
\subsection{About this paper}
In this paper, we will study asymptotics of Toeplitz matrices with symbols
in generalized H\"older spaces following the approach of \cite{BS80}.
Our results improve earlier results by B\"ottcher and Silbermann for
$C_{N\times N}^\gamma$, $0<\gamma<1$, because the scale of generalized H\"older
spaces is finer than the scale of H\"older spaces $C^\gamma$, $0<\gamma<1$
(although we will not consider generalizations of the case $\gamma\ge 1$).

The paper is organized as follows. Section~\ref{sec:preliminaries} contains
definitions of Schatten-von Neumann classes and regularized operator
determinants, as well as definitions of the Bari-Stechkin class and generalized
H\"older spaces $\cH^\omega$ and their subspaces $\cH_0^\omega$. Our main results
refining Theorems~\ref{th:Widom1} and \ref{th:Widom2} are stated in the end of
Section~\ref{sec:preliminaries}. In Section~\ref{sec:abstract}, we present an
abstract approach from \cite{BS80} (see also \cite{K-JMAA}) to higher order
asymptotic formulas for block Toeplitz matrices. To apply these results it
is necessary to check that the symbol admits canonical left and right bounded
Wiener-Hopf factorizations, at least one of the factors is continuous, and some
products of Hankel operators belong to the Schatten-von Neumann class
$\cC_m(H_N^2)$ for $m\in\N$. In Section~\ref{sec:factorization}, we collect
necessary information about Wiener-Hopf factorization in decomposing algebras
of continuous functions and verify that the algebras $(\cH^\omega)_{N\times N}$
and $(\cH_0^\omega)_{N\times N}$ have the factorization property. In
Section~\ref{sec:approximation}, we prove simple sufficient conditions for the
membership in the Schatten-von Neumann classes of products of Hankel operators
with symbols in $(\cH^\omega)_{N\times N}$. These results are based on the
classical Jackson theorem on the best uniform approximation. In
Section~\ref{sec:proofs}, we prove our asymptotic formulas on the basis of the
results of Sections~\ref{sec:abstract}--\ref{sec:approximation}.
\section{Preliminaries and the main results}
\label{sec:preliminaries}
\subsection{Schatten-von Neumann classes and operator determinants}
\label{sec:Schatten-Neumann}
Let $H$ be a separable Hilbert space, $\cB(H)$ be the Banach algebra of all
bounded linear operators on $H$, $\cC_0(H)$ be the set of all finite-rank
operators, and $\cC_\infty(H)$ be the closed two-sided ideal of all compact
operators on $H$. Given $A\in\cB(H)$ define
$s_n(A):=\inf\{\|A-F\|_{\cB(H)}: F\in\cC_0(H),\ \dim F(H)\le n\}$
for $n\in\Z_+$.
For $1\le p<\infty$, the collection of all operators $K\in\cB(H)$ satisfying
\[
\|K\|_{\cC_p(H)}:=\Bigg(\sum_{n\in\Z_+}s_n^p(K)\Bigg)^{1/p}<\infty
\]
is denoted by $\cC_p(H)$ and referred to as a \textit{Schatten-von Neumann class}.
Note that $\cC_\infty(H)=\{K\in\cB(H):s_n(K)\to 0\mbox{ as }n\to\infty\}$ and
\[
\|K\|_{\cC_\infty(H)}=\sup_{n\in\Z_+}s_n(K)=\|K\|_{\cB(H)}.
\]
The operators belonging to $\cC_1(H)$ are called \textit{trace class operators}.

Let $A\in\cB(H)$ be an operator of the form $I+K$ with $K\in\cC_1(H)$. If
$\{\lambda_j(K)\}_{j\ge 0}$ denotes the sequence of the nonzero eigenvalues of
$K$ counted up to algebraic multiplicity, then the product
$\prod_{j\ge 0}(1+\lambda_j(K))$ is absolutely convergent. The \textit{determinant}
of $A$ is defined by
\[
\det A=\det(I+K)=\prod_{j\ge 0}(1+\lambda_j(K)).
\]
If $K\in\cC_m(H)$, where $m\in\N\setminus\{1\}$, one can still define a determinant
of $I+K$, but for classes larger than $\cC_1(H)$, the above definition requires a
regularization. A simple computation (see \cite[Lemma~6.1]{Simon77}) shows that then
\[
R_m(K):=(I+K)\exp\Bigg(\sum_{j=1}^{m-1}\frac{(-K)^j}{j}\Bigg)-I\in\cC_1(H).
\]
Thus, it is natural to define
\[
\mathrm{det}_1(I+K):=\det(I+K),
\quad
\mathrm{det}_m(I+K):=\det(I+R_m(K))
\ \mbox{for}\ m\in\N\setminus\{1\}.
\]
One calls $\det_m(I+K)$ the $m$-\textit{regularized determinant} of $A=I+K$.
For more information about Schatten-von Neumann classes and regularized operator
determinants, see \cite[Chap. III--IV]{GK69} and also \cite{Simon77}.
\subsection{The Bari-Stechkin class}
A real-valued function $\varphi$ is said to be \textit{almost increasing}
on an interval $I$ of $\R$ if there is a positive constant $A$ such that
$\varphi(x)\le A\varphi(y)$ for all $x,y\in I$ such that $x\le y$.
One says that $\omega:(0,\pi]\to[0,\infty)$ belongs to the \textit{Bari-Stechkin class}
(see \cite[p.~493]{BS56} and \cite[Chap.~2, Section~2]{GM80})
if $\omega$ is almost increasing on $(0,\pi]$, $\omega(x)>0$ for all $x\in(0,\pi]$, and
\[
\lim_{x\to 0+0}\omega(x)=0,
\quad
\sup_{x>0}\frac{1}{\omega(x)}\int_0^x\frac{\omega(y)}{y}\,dy<\infty,
\quad
\sup_{x>0}\frac{x}{\omega(x)}\int_x^\pi\frac{\omega(y)}{y^2}\,dy<\infty.
\]

To give an example of functions in the Bari-Stechkin class, let us define
inductively the sequence of functions $\ell_k$ on $(x_k,\infty)$ by
$\ell_1(x):=\log x$, $x_1:=1$ and for $k\in\N\setminus\{1\}$,
$\ell_k(x):=\log(\ell_{k-1}(x))$ and $x_k$ such that $\ell_{k-1}(x_k)=1$.
Elementary computations show that for every $\gamma\in(0,1)$ and any finite
sequence $\beta_1,\dots,\beta_m\in\R$ there exists a set of positive constants
$b_1,\dots,b_m$ such that the function
\begin{equation}\label{eq:Bari-Stechkin-example}
\omega(x)=x^\gamma\prod_{k=1}^m\ell_k^{\beta_k}\left(\frac{b_k}{x}\right),
\quad
0<x\le\pi,
\end{equation}
belongs to the Bari-Stechkin class. In particular, $\omega(x)=x^\gamma$,
$0<\gamma<1$, is a trivial example of a function in the Bari-Stechkin class.
\subsection{Generalized H\"older spaces}
The \textit{modulus of continuity} of a bounded function $f:\T\to\C$ is defined
by
\[
\omega(f,x):=\sup_{|h|\le x}\sup_{y\in\R}|f(e^{i(y+h)})-f(e^{iy})|,
\quad 0\le x\le \pi.
\]

Let $\omega$ belong to the Bari-Stechkin class. The generalized H\"older space
$\cH^\omega$ is defined as the set of all continuous functions $f:\T\to\C$
satisfying
\[
|f|_\omega:=\sup_{0<x\le \pi}\frac{\omega(f,x)}{\omega(x)}<\infty.
\]
We will consider also the subspace $\cH_0^\omega$ of functions $f\in\cH^\omega$
such that
\[
\lim_{x\to 0+0}\frac{\omega(f,x)}{\omega(x)}=0.
\]
It is well known that $\cH^\omega$ and $\cH_0^\omega$ are Banach algebras
under the norm
\[
\|f\|_{\cH^\omega}:=\|f\|_C+|f|_\omega.
\]
\subsection{Higher order asymptotic formulas for determinants}
 For $a\in L_{N\times N}^\infty$ and $n\in\Z_+$,
define the operators $P_n$ and $Q_n$ on $H_N^2$ by
\[
P_n:\sum_{k=0}^\infty a_k t^k\mapsto \sum_{k=0}^na_k t^k,
\quad
Q_n:=I-P_n.
\]
The operator $P_nT(a)P_n:P_nH_N^2\to P_nH_N^2$ may be identified with the finite
block Toeplitz matrix $T_n(a):=(a_{j-k})_{j,k=0}^n$.

If $\cA$ is a unital algebra, then its group of all invertible elements is denoted
by $G\cA$. For $1\le p\le\infty$, put $\overline{H^p}:=\{f\in L^p:\overline{f}\in H^p\}$.
Suppose
\begin{eqnarray}
v_-\in (\overline{H^\infty})_{N\times N},
&&
v_+\in H_{N\times N}^\infty,
\label{eq:factors1}
\\
u_-\in G(\overline{H^\infty})_{N\times N},
&&
u_+\in GH_{N\times N}^\infty,
\label{eq:factors2}
\end{eqnarray}
and define
\[
b:=v_-u_+^{-1},
\quad
c:=u_-^{-1}v_+.
\]
\begin{theorem}[Main result 1]
\label{th:HO1}
Let $\omega,\psi$ belong to the Bari-Stechkin class.
Suppose $a\in L_{N\times N}^\infty$ can be factored as $a=u_-u_+$ with
\begin{equation}\label{eq:factors3}
u_-\in G(\cH^\omega\cap \overline{H^\infty})_{N\times N},
\quad
u_+\in G(\cH^\psi\cap H^\infty)_{N\times N},
\end{equation}
and suppose $T(\widetilde{a})$ is invertible on $H_N^2$. Then the following
statements hold.
\begin{enumerate}
\item[(a)]
The function $a$ admits a factorization $a=v_+v_-$, where
$v_-\in G(\overline{H^\infty})_{N\times N}$ and $v_+\in GH_{N\times N}^\infty$.

\item[(b)]
If
\begin{equation}\label{eq:HO1-1}
\sum_{k=1}^\infty \omega\left(\frac{1}{k}\right)\psi\left(\frac{1}{k}\right)<\infty,
\end{equation}
then $T(a)T(a^{-1})-I$ and $T(\widetilde{c})T(\widetilde{b})-I$ belong to $\cC_1(H_N^2)$
and
\[
\lim_{n\to\infty}\frac{\det T_n(a)}{G(a)^{n+1}}
=\mathrm{det}_1T(a)T(a^{-1})
=\frac{1}{\mathrm{det}_1 T(\widetilde{c})T(\widetilde{b})}.
\]

\item[(c)]
If $m\in\N\setminus\{1\}$ and
\begin{equation}\label{eq:HO1-2}
\sum_{k=1}^\infty
\left[\omega\left(\frac{1}{k}\right)\psi\left(\frac{1}{k}\right)\right]^m<\infty,
\end{equation}
then $T(\widetilde{c})T(\widetilde{b})-I\in\cC_m(H_N^2)$ and
\begin{equation}\label{eq:HO1-3}
\lim_{n\to\infty}\frac{\det T_n(a)}{G(a)^{n+1}}
\exp\left\{
-\sum_{j=1}^{m-1}\frac{1}{j}\tr\left[\left(\sum_{k=0}^{m-1}F_{n,k}(b,c)\right)^j\right]
\right\}
=\frac{1}{\mathrm{det}_m T(\widetilde{c})T(\widetilde{b})},
\end{equation}
where
\[
F_{n,k}(b,c):=P_nT(c)Q_n\big(Q_nH(b)H(\widetilde{c})Q_n\big)^kQ_nT(b)P_n
\quad (n,k\in\Z_+).
\]

\item[(d)]
Suppose $m\in\N\setminus\{1\}$. If \eqref{eq:HO1-2} is fulfilled and
\begin{equation}\label{eq:HO1-4}
\lim_{n\to\infty}
\Bigg\{
\left[\omega\left(\frac{1}{n}\right)\psi\left(\frac{1}{n}\right)\right]^{m-1}
\sum_{j=1}^n
\omega\left(\frac{1}{j}\right)\psi\left(\frac{1}{j}\right)
\Bigg\}
=0,
\end{equation}
then one can remove $F_{n,m-1}(b,c)$ in \eqref{eq:HO1-3}, that is,
\begin{equation}\label{eq:HO1-5}
\lim_{n\to\infty}\frac{\det T_n(a)}{G(a)^{n+1}}
\exp\left\{
-\sum_{j=1}^{m-1}\frac{1}{j}\tr\left[\left(\sum_{k=0}^{m-2}F_{n,k}(b,c)\right)^j\right]
\right\}
=\frac{1}{\mathrm{det}_m T(\widetilde{c})T(\widetilde{b})}.
\end{equation}

\item[(e)]
If $m\in\N$ and \eqref{eq:HO1-2} is fulfilled, then there
exists a nonzero constant $E(a)$ such that
\begin{equation}\label{eq:HO1-6}
\begin{split}
\log\det T_n(a) &= (n+1)\log G(a)+\log E(a)
\\
&\quad+\tr
\left[\sum_{\ell=1}^n\sum_{j=1}^{m-1}\frac{1}{j}\left(
\sum_{k=0}^{m-j-1}G_{\ell,k}(b,c)
\right)^j\right]
\\
&\quad+
O\left(\sum_{k=n+1}^\infty
\left[\omega\left(\frac{1}{k}\right)\psi\left(\frac{1}{k}\right)\right]^m\right)
\end{split}
\end{equation}
as $n\to\infty$, where
\[
G_{\ell,k}(b,c):=
P_0T(c)Q_\ell\big(Q_\ell H(b)H(\widetilde{c})Q_\ell\big)^kQ_\ell T(b)P_0
\quad(\ell,k\in\Z_+).
\]

\item[(f)]
If, under the assumptions of part {\rm(e)},
\[
u_-\in G(\cH_0^\omega\cap \overline{H^\infty})_{N\times N}
\quad\mbox{or}\quad
u_+\in G(\cH_0^\psi\cap H^\infty)_{N\times N},
\]
then \eqref{eq:HO1-6} holds with $O(\dots)$ replaced by $o(\dots)$.
\end{enumerate}
\end{theorem}
Let $\alpha,\beta\in(0,1)$ and $\omega(x)=x^\alpha$, $\psi(x)=x^\beta$.
If $\alpha+\beta>1$, then \eqref{eq:HO1-1} holds. If $\alpha+\beta>1/m$
for some $m\in\N\setminus\{1\}$, then \eqref{eq:HO1-2} and \eqref{eq:HO1-4}
are fulfilled and we arrive at the theorem of B\"ottcher and Silbermann
\cite[Theorems~10.35(ii) and 10.37(ii)]{BS06} for standard H\"older spaces.
It seems that part (f) is new even for standard H\"older spaces.
\subsection{Refinements of the Szeg\H{o}-Widom limit theorems}
The case of $\omega=\psi$  in Theorem~\ref{th:HO1} is of particular importance.
In this case we will prove the following refinement of the Szeg\H{o}-Widom limit
theorems.
\begin{theorem}[Main result 2]
\label{th:HO2}
Let $\omega$ belong to the Bari-Stechkin class and let $\cH$ be either $\cH^\omega$
or $\cH_0^\omega$. Suppose
\begin{equation}\label{eq:HO2-1}
\sum_{k=1}^\infty \left[\omega\left(\frac{1}{k}\right)\right]^2<\infty
\end{equation}
and put
\[
\delta(n,\cH)
:=
\left\{
\begin{array}{lll}
\displaystyle
O\left(\sum_{k=n+1}^\infty\left[\omega\left(\frac{1}{k}\right)\right]^2\right)
&\mbox{if}&\cH=\cH^\omega,
\\
\displaystyle
o\left(\sum_{k=n+1}^\infty\left[\omega\left(\frac{1}{k}\right)\right]^2\right)
&\mbox{if}&\cH=\cH_0^\omega.
\end{array}
\right.
\]
\begin{enumerate}
\item[(a)]
We have $\cH_{N\times N}\subset (K_{2,2}^{1/2,1/2})_{N\times N}$.

\item[(b)]
If $a\in\cH_{N\times N}$ and the Toeplitz operators
$T(a)$ and $T(\widetilde{a})$ are invertible on $H_N^2$, then
\eqref{eq:Widom1} holds with $o(1)$ replaced by $\delta(n,\cH)$.

\item[(c)]
If $a\in\cH_{N\times N}$ and $f$ is analytic on
$\mathrm{sp}\,T(a)\cup\mathrm{sp}\,T(\widetilde{a})$, then
\eqref{eq:Widom2} holds with $o(1)$ replaced by $\delta(n,\cH)$.
\end{enumerate}
\end{theorem}
For $\cH^\omega=C^\gamma$ with $\gamma\in(1/2,1)$ and $O(n^{1-2\gamma})$
in place of $\delta(n,\cH)$, parts (a) and (b) are already in \cite{BS80}
(see also \cite{BW06}) and part (c) is in \cite{K-IWOTA05}. Notice that the
scale of generalized H\"older spaces is finer than the scale of standard
H\"older spaces. For instance, for every $\gamma\in(0,1)$ there exist functions
$\omega_1$ and $\omega_2$ of the form \eqref{eq:Bari-Stechkin-example} such that
\[
\bigcup_{0<\eps<1-\gamma}C^{\gamma+\eps}\subset
\cH^{\omega_1}\subset
C^\gamma\subset
\cH^{\omega_2}\subset
\bigcap_{0<\eps<\gamma}C^{\gamma-\eps},
\]
where each of the embeddings is proper (see \cite[Section~II.3]{GM80}).
Hence, Theorems~\ref{th:HO1} and \ref{th:HO2} refine corresponding results for
standard H\"older spaces.
\section{Higher order asymptotic formulas:\\ the approach of B\"ottcher and Silbermann}
\label{sec:abstract}
\subsection{Asymptotic formulas involving regularized operator determinants}
The following result goes back to B\"ottcher and Silbermann \cite{BS80}
(see also \cite[Sections 6.15 and 6.20]{BS83} and \cite[Sections 10.34 and 10.37]{BS06}).
\begin{theorem}\label{th:HO-abstract}
Suppose $a\in L_{N\times N}^\infty$ satisfies the following assumptions:
\begin{enumerate}
\item[(i)]
there are two factorizations $a=u_-u_+=v_+v_-$, where
$u_-,v_-\in G(\overline{H^\infty})_{N\times N}$ and $u_+,v_+\in GH_{N\times N}^\infty$;

\item[(ii)]
$u_-\in C_{N\times N}$ or $u_+\in C_{N\times N}$.
\end{enumerate}
Then the following statements are true.
\begin{enumerate}
\item[(a)]
If $H(a)H(\widetilde{a}^{-1})\in\cC_1(H_N^2)$, then
\[
\lim_{n\to\infty}\frac{\det T_n(a)}{G(a)^{n+1}}=\mathrm{det}_1 T(a)T(a^{-1}).
\]

\item[(b)]
If $H(b)H(\widetilde{c})$ and $H(\widetilde{c})H(b)$ belong to $\cC_m(H_N^2)$
for some $m\in\N$, then \eqref{eq:HO1-3} is fulfilled.

\item[(c)]
If $H(b)H(\widetilde{c})$ and $H(\widetilde{c})H(b)$ belong to $\cC_m(H_N^2)$
for some $m\in\N\setminus\{1\}$ and
\[
\lim_{n\to\infty}\tr F_{n,m-1}(b,c)=0,
\]
then \eqref{eq:HO1-5} holds.
\end{enumerate}
\end{theorem}
\begin{proof}
Part (a) follows from \cite[Corollary~10.27]{BS06}. Part (b) is proved in the
present form in \cite[Theorem~15]{K-JMAA}. Part (c) follows from part (b)
and \cite[Propositions~6,13, and 14]{K-JMAA}.
\end{proof}
Notice that hypothesis (ii) can be replaced by a weaker hypothesis
(see \cite[Section~10.34]{BS06}), which allows
us to work with two discontinuous factors $u_-$ and $u_+$. This is useful
in the case of generalized Krein algebras $(K_{p,q}^{\alpha,\beta})_{N\times N}$
(see \cite{BKS-Krein100}).
\subsection{Decomposition of the logarithm of Toeplitz determinants}
The following lemma is an important step in the proof of
Theorems~\ref{th:HO1}(e), (f) and Theorem~\ref{th:HO2}. It was obtained in
\cite{BS80} (see also \cite[Section~6.16]{BS83} and \cite[Section~10.34]{BS06}).
\begin{lemma}\label{le:BS}
Suppose $a\in L_{N\times N}^\infty$ satisfies hypotheses {\rm(i)} and {\rm(ii)}
of Theorem~{\rm\ref{th:HO-abstract}}. Suppose for all sufficiently large
$n$ (say, $n\ge n_0$) there exists a decomposition
\[
\tr\log\left\{I-\sum_{k=0}^\infty G_{n,k}(b,c)\right\}
=-\tr\, M_n+s_n,
\]
where $\{M_n\}_{n=n_0}^\infty$ is a sequence of $N\times N$ matrices and
$\{s_n\}_{n=n_0}^\infty$ is a sequence of complex numbers. If
$\sum_{n=n_0}^\infty |s_n|<\infty$, then there exists a constant $E(a)\ne 0$,
depending on $\{M_n\}_{n=n_0}^\infty$ and arbitrarily chosen $N\times N$ matrices
$M_1,\dots,M_{n_0-1}$, such that for all $n\ge n_0$,
\[
\log\det T_n(a)=(n+1)\log G(a)+\tr(M_1+\dots M_n)+\log E(a)+\sum_{k=n+1}^\infty s_k.
\]
\end{lemma}
\section{Wiener-Hopf factorization in decomposing algebras of continuous functions}
\label{sec:factorization}
\subsection{Definitions and general theorems}
Let $\cA$ be a Banach algebra continuously embedded in $C$. Suppose $\cA$
contains the set of all rational functions without poles on $\T$ and $\cA$
is inverse closed in $C$, that is, if $a\in\cA$ and $a(t)\ne 0$ for all $t\in\T$,
then $a^{-1}\in\cA$. The sets $\cA_-:=\cA\cap\overline{H^\infty}$ and
$\cA_+:=\cA\cap H^\infty$ are subalgebras of $\cA$.
The algebra $\cA$ is said to be \textit{decomposing} if every function $a\in\cA$
can be represented in the form $a=a_-+a_+$ where $a_\pm\in\cA_\pm$.

Let $\cA$ be a decomposing algebra. A matrix function
$a\in\cA_{N\times N}$ is said to admit a \textit{right} (resp. \textit{left})
\textit{Wiener-Hopf factorization in} $\cA_{N\times N}$ if it can
be represented in the form $a=a_-Da_+$ (resp. $a=a_+Da_-$), where
\[
a_\pm\in G(\cA_\pm)_{N\times N},
\quad
D(t)=\mathrm{diag}\{t^{\kappa_1},\dots,t^{\kappa_N}\},
\quad
\kappa_i\in\Z,
\quad
\kappa_1\le\dots\le\kappa_N.
\]
The integers $\kappa_i$ are usually called
the \textit{right} (resp. \textit{left}) \textit{partial indices} of $a$;
they can be shown to be uniquely determined by $a$. If $\kappa_1=\dots=\kappa_N=0$,
then the respective Wiener-Hopf factorization is said to be \textit{canonical}.

The following result was obtained by Budjanu and Gohberg \cite[Theorem~4.3]{BG68}
and it is contained in \cite[Chap.~II, Corollary~5.1]{CG81} and in
\cite[Theorem~5.7']{LS87}.
\begin{theorem}\label{th:factorization}
Suppose the following two conditions hold for the algebra $\cA$:
\begin{enumerate}
\item[{\rm (a)}]
the Cauchy singular integral operator
\[
(S\varphi)(t):=\frac{1}{\pi i}v.p.\int_\T\frac{\varphi(\tau)}{\tau-t}d\tau
\quad(t\in\T)
\]
is bounded on $\cA$;

\item[{\rm (b)}]
for any function $a\in\cA$, the operator $aS-SaI$ is compact on $\cA$.
\end{enumerate}
Then every matrix function $a\in\cA_{N\times N}$ such that $\det a(t)\ne 0$
for all $t\in\T$ admits a right Wiener-Hopf factorization in $\cA_{N\times N}$.
\end{theorem}
Notice that (a) holds if and only if $\cA$ is a decomposing algebra.

The following theorem follows from a more general result due to Shubin
(see \cite[Theorem~6.15]{LS87}).
\begin{theorem}\label{th:stability}
Let $\cA$ be a decomposing algebra and let $\|\cdot\|$ be a norm in
the algebra $\cA_{N\times N}$. Suppose $a,d\in\cA_{N\times N}$ admit canonical
right and left Wiener-Hopf factorizations in the algebra $\cA_{N\times N}$. Then for
every $\eps>0$ there exists a $\delta>0$ such that if $\|a-d\|<\delta$, then for
every canonical right Wiener-Hopf factorization $a=a_-^{(r)}a_+^{(r)}$ and for every
canonical left Wiener-Hopf factorization $a=a_+^{(l)}a_-^{(l)}$ one can choose
a canonical right Wiener-Hopf factorization $d=d_-^{(r)}d_+^{(r)}$ and a canonical
left Wiener-Hopf factorization $d=d_+^{(l)}d_-^{(l)}$ such that
\[
\begin{array}{lll}
\|a_\pm^{(r)}-d_\pm^{(r)}\|<\eps,
&\quad&
\|[a_\pm^{(r)}]^{-1}-[d_\pm^{(r)}]^{-1}\|<\eps,
\\[3mm]
\|a_\pm^{(l)}-d_\pm^{(l)}\|<\eps,
&\quad&
\|[a_\pm^{(l)}]^{-1}-[d_\pm^{(l)}]^{-1}\|<\eps.
\end{array}
\]
\end{theorem}
\subsection{Verification of the hypotheses of Theorem~\ref{th:factorization}
for generalized H\"older spaces}
\begin{theorem}\label{th:verification}
Let $\omega$ belong to the Bari-Stechkin class and let $\cH$ be either $\cH^\omega$
or $\cH_0^\omega$. Then
\begin{enumerate}
\item[{\rm (a)}] $a\in \cH$ is invertible in $\cH$ if and only if
$a(t)\ne 0$ for all $t\in\T$;
\item[{\rm (b)}] $S$ is bounded on $\cH$;
\item[{\rm (c)}] for $a\in\cH$, the operator $aS-SaI$ is compact on $\cH$.
\end{enumerate}
\end{theorem}
\begin{proof}
(a) Obviously, $G\cH\subset GC$. Conversely, if $f\in GC$, then for all
$x\in(0,\pi]$,
\[
\omega(1/f,x)\le\|1/f\|_C^2\,\omega(f,x).
\]
From this inequality we see that if $f\in GC\cap\cH$, then $1/f\in\cH$.
Part (a) is proved.

(b) For $\cH^\omega$ this result follows from the well known Zygmund estimate
(see \cite{Zygmund24} and also \cite[p.~492]{BS56}, \cite[p.~10]{GM80}):
\begin{equation}\label{eq:Zygmund-estimate}
\omega(Sf,x)\le
c\int_0^x\frac{\omega(f,y)}{y}\,dy+
cx\int_x^\pi\frac{\omega(f,y)}{y^2}\,dy,
\quad
0<x\le\pi,
\end{equation}
with a positive constant $c$ independent of $f\in\cH^\omega$. For a self-contained
proof of the boundedness of $S$ on $\cH^\omega$ (in a more general situation of
moduli of smoothness $\omega_\alpha(f,x)$ of order $\alpha>0$), see
S. Samko and A. Yakubov \cite[Theorem~2]{SYa85}.

If $f\in\cH_0^\omega$ and
\[
F_1(f,x):=\int_0^x\frac{\omega(f,y)}{y}\,dy,
\quad
F_2(f,x):=x\int_x^\pi\frac{\omega(f,y)}{y^2}\,dy,
\]
then, by \cite[Section~IV.4, Lemma~1]{GM80},
\begin{equation}\label{eq:verification}
\lim_{x\to 0+0}\frac{F_1(f,x)}{\omega(x)}
=
\lim_{x\to 0+0}\frac{F_2(f,x)}{\omega(x)}
=0.
\end{equation}
From \eqref{eq:Zygmund-estimate} and \eqref{eq:verification} it follows that
$Sf\in\cH_0^\omega$ whenever $f\in\cH_0^\omega$. That is, $S$ is bounded
on $\cH_0^\omega$, too. Part (b) is proved.

(c) For $a\in\cH^\omega$, the compactness of $aS-SaI$ on $\cH^\omega$ was
proved by Tursunkulov \cite{Tursunkulov82} (see also a survey by
N.~Samko \cite[Corollary~4.8]{Samko04}).

If $a\in\cH_0^\omega$, then $aS-SaI$ is bounded on $\cH_0^\omega$ by part (b)
and is compact on $\cH^\omega$ by what has just been said above. Since
$\cH_0^\omega\subset\cH^\omega$, it is easy to see that the operator $aS-SaI$
is also compact on $\cH_0^\omega$.
\end{proof}
\subsection{Wiener-Hopf factorization in generalized H\"older spaces}
\begin{theorem}
\label{th:factorization-Hoelder}
Let $\omega$ belong to the Bari-Stechkin class and let $\cH$ be either
$\cH^\omega$ or $\cH_0^\omega$. Suppose $a\in\cH_{N\times N}$.
\begin{enumerate}
\item[{\rm (a)}]
If $T(a)$ is invertible on $H_N^2$, then $a$ admits a canonical right Wiener-Hopf
factorization in $\cH_{N\times N}$.

\item[{\rm (b)}]
If $T(\widetilde{a})$ is invertible on $H_N^2$, then $a$ admits a canonical
left Wiener-Hopf factorization in $\cH_{N\times N}$.
\end{enumerate}
\end{theorem}
\begin{proof}
We follow the proof of \cite[Theorem~2.4]{K-IWOTA05}.

(a) If $T(a)$ is invertible on $H_N^2$, then $\det a(t)\ne 0$ for all $t\in\T$
(see, e.g., \cite[Chap.~VII, Proposition~2.1]{CG81}). Then, by
\cite[Chap.~VII, Theorem~3.2]{CG81}, the matrix function $a$ admits a canonical
right generalized factorization in $L_N^2$, that is, $a=a_-a_+$, where
$a_-^{\pm 1}\in (\overline{H^2})_{N\times N}$,
$a_+^{\pm 1}\in H^2_{N\times N}$ (and, moreover, the operator
$a_-Pa_-^{-1}I$ is bounded on $L_N^2$).

On the other hand, from Theorems~\ref{th:factorization} and \ref{th:verification}
it follows that $a\in\cH_{N\times N}$ admits a right Wiener-Hopf factorization
$a=u_-Du_+$ in $\cH_{N\times N}$. Then
\[
u_-^{\pm 1}\in (\cH_-)_{N\times N}\subset(\overline{H^2})_{N\times N},
\quad
u_+^{\pm 1}\in (\cH_+)_{N\times N}\subset H^2_{N\times N},
\]
that is, $a=u_-Du_+$ is a right generalized factorization in $L_N^2$.
By the uniqueness of the partial indices in a right generalized factorization
in $L_N^2$ (see, e.g., \cite[Corollary~2.1]{LS87}), $D=1$. Part (a) is proved.

(b) In view of Theorem~\ref{th:verification}(a), $a^{-1}\in\cH_{N\times N}$.
By \cite[Proposition~7.19(b)]{BS06}, the invertibility of $T(\widetilde{a})$ on
$H_N^2$ is equivalent to the invertibility of $T(a^{-1})$ on $H_N^2$. In view of
part (a), there exist $f_\pm\in G(\cH_\pm)_{N\times N}$ such that $a^{-1}=f_-f_+$.
Put $v_\pm:=f_\pm^{-1}$. Then $v_\pm\in G(\cH_\pm)_{N\times N}$ and
$a=v_+v_-$ is a canonical left Wiener-Hopf factorization in $\cH_{N\times N}$.
\end{proof}
\section{Some applications of approximation theory}
\label{sec:approximation}
\subsection{The best uniform approximation}
For $n\in\Z_+$, let $\cP^n$ be the set of all Laurent polynomials of the form
\[
p(t)=\sum_{j=-n}^n\alpha_j t^j,\quad\alpha_j\in\C,\quad t\in\T.
\]
By the Chebyshev theorem (see, e.g., \cite[Section~2.2.1]{Timan63}),
for $f\in C$ and $n\in\Z_+$, there is a Laurent polynomial $p_n(f)\in\cP^n$
such that
\begin{equation}\label{eq:approx1}
\|f-p_n(f)\|_C=\inf_{p\in\cP^n}\|f-p\|_C.
\end{equation}
Each such polynomial $p_n(f)$ is called a polynomial of best uniform
approximation.

By the Jackson theorem (see, e.g., \cite[Section~5.1.2]{Timan63}),
there exists a constant $A>0$ such that for all $f\in C$ and
all $n\in\Z_+$,
\begin{equation}\label{eq:approx2}
\inf_{p\in\cP^n}\|f-p\|_C \le A \omega\left(f,\frac{1}{n+1}\right).
\end{equation}
\subsection{Norms of truncations of Toeplitz and Hankel operators}
Let $X$ be a Banach space. For definiteness, let the norm of
$a=(a_{\alpha,\beta})_{\alpha,\beta=1}^N$ in $X_{N\times N}$ be given by
\[
\|a\|_{X_{N\times N}}:=N\max\limits_{1\le \alpha,\beta\le N}\|a_{\alpha,\beta}\|_X.
\]
We will simply write $\|a\|_\infty$, $\|a\|_C$, and $\|a\|_\omega$ instead of
$\|a\|_{L_{N\times N}^\infty}$, $\|a\|_{C_{N\times N}}$, and
$\|a\|_{(\cH^\omega)_{N\times N}}$, respectively.
Denote by $\|A\|$ the norm of a bounded linear operator $A$ on $H_N^2$.

Put $\Delta_0:=P_0$ and $\Delta_j:=P_j-P_{j-1}$ for $j\in\{0,\dots,n\}$.
\begin{lemma}\label{le:trunc}
Let $n\in\Z_+$. Suppose $v_\pm$ and $u_\pm$ satisfy \eqref{eq:factors1},
\eqref{eq:factors2}, and $u_\pm^{-1}\in C_{N\times N}$. Then there exists
a positive constant $A_N$ depending only on $N$ such that for all
$n\in\Z_+$ and all $j\in\{0,\dots,n\}$,
\begin{eqnarray}
\|Q_nT(b)\Delta_j\|
&\le&
A_N\|v_-\|_\infty\max_{1\le\alpha,\beta\le N}
\omega\left([u_+^{-1}]_{\alpha,\beta},\frac{1}{n-j+1}\right),
\label{eq:trunc-1}
\\
\|\Delta_jT(c)Q_n\|
&\le&
A_N\|v_+\|_\infty\max_{1\le\alpha,\beta\le N}
\omega\left([u_-^{-1}]_{\alpha,\beta},\frac{1}{n-j+1}\right),
\label{eq:trunc-2}
\\
\|Q_nH(b)\|
&\le&
A_N\|v_-\|_\infty\max_{1\le\alpha,\beta\le N}
\omega\left([u_+^{-1}]_{\alpha,\beta},\frac{1}{n+1}\right),
\label{eq:trunc-3}
\\
\|H(\widetilde{c})Q_n\|
&\le&
A_N\|v_+\|_\infty\max_{1\le\alpha,\beta\le N}
\omega\left([u_-^{-1}]_{\alpha,\beta},\frac{1}{n+1}\right).
\label{eq:trunc-4}
\end{eqnarray}
\end{lemma}
\begin{proof}
The idea of the proof is borrowed from \cite[Theorem~10.35(ii)]{BS06}
(see also \cite[Proposition~3.2]{K-IWOTA05}). Since $u_+^{-1},v_+\in H_{N\times N}^\infty$
and $u_-^{-1},v_-\in (\overline{H^\infty})_{N\times N}$, by \cite[Proposition~2.14]{BS06},
\begin{eqnarray*}
T(b)=T(v_-)T(u_+^{-1}),
&\quad&
T(c)=T(u_-^{-1})T(v_+),
\\
H(b)=T(v_-)H(u_+^{-1}),
&\quad&
H(\widetilde{c})=H(\widetilde{u_-^{-1}})T(v_+).
\end{eqnarray*}
It is easy to see that $Q_nT(v_-)P_n=0$ and $P_nT(v_+)Q_n=0$. Hence
\begin{eqnarray}
Q_nT(b)\Delta_j
&=&
Q_nT(v_-)Q_nT(u_+^{-1})\Delta_j,
\label{eq:trunc-5}
\\
\Delta_j T(c)Q_n
&=&
\Delta_j T(u_-^{-1})Q_n T(v_+)Q_n,
\label{eq:trunc-6}
\\
Q_nH(b)
&=&
Q_nT(v_-)Q_nH(u_+^{-1}),
\label{eq:trunc-7}
\\
H(\widetilde{c})Q_n
&=&
H(\widetilde{u_-^{-1}})Q_nT(v_+)Q_n.
\label{eq:trunc-8}
\end{eqnarray}
Let $p_{n-j}(u_+^{-1})$ and $p_{n-j}(u_-^{-1})$ be the polynomials in
$\cP_{N\times N}^{n-j}$ of best uniform approximation of $u_+^{-1}$ and
$u_-^{-1}$, respectively, where $j\in\{0,\dots,n\}$. Simple computations
show that
\begin{equation}\label{eq:trunc-9}
Q_nT[p_{n-j}(u_+^{-1})]\Delta_j=0,
\quad
\Delta_jT[p_{n-j}(u_-^{-1})]Q_n=0
\end{equation}
for all $j\in\{0,\dots,n\}$ and
\begin{equation}\label{eq:trunc-10}
Q_nH[p_n(u_+^{-1})]=0,
\quad
H[(p_n(u_-^{-1}))\widetilde{\hspace{2mm}}]Q_n=0.
\end{equation}
From \eqref{eq:trunc-5} and \eqref{eq:trunc-9} we get
\begin{eqnarray*}
\|Q_nT(b)\Delta_j\|
&\le&
\|Q_nT(v_-)Q_n\|\,
\|Q_nT[u_+^{-1}-p_{n-j}(u_+^{-1})]\Delta_j\|
\\
&\le&
\mathrm{const}\,\|v_-\|_\infty\, \|u_+^{-1}-p_{n-j}(u_+^{-1})\|_C.
\end{eqnarray*}
Combining this inequality with \eqref{eq:approx1}--\eqref{eq:approx2}, we
arrive at \eqref{eq:trunc-1}. Inequalities \eqref{eq:trunc-2}--\eqref{eq:trunc-4}
can be obtained in the same way by combining \eqref{eq:trunc-9}--\eqref{eq:trunc-10}
and representations \eqref{eq:trunc-6}--\eqref{eq:trunc-8}, respectively.
\end{proof}
\subsection{The asymptotic of the trace of $F_{n,m-1}(b,c)$}
\begin{lemma}\label{le:traceF}
Let $\omega,\psi$ belong to the Bari-Stechkin class. Suppose $v_\pm$ and $u_\pm$
satisfy \eqref{eq:factors1} and \eqref{eq:factors3}. If $m\in\N\setminus\{1\}$
and \eqref{eq:HO1-4} is fulfilled, then $\tr F_{n,m-1}(b,c)\to 0$ as $n\to\infty$.
\end{lemma}
\begin{proof}
Since $\Delta_j F_{n,m-1}(b,c)\Delta_j$ is an $N\times N$ matrix for each
$n\in\Z_+$ and each $j\in\{0,\dots,n\}$, we have
\[
|\tr \Delta_j F_{n,m-1}(b,c)\Delta_j|\le C_N\|\Delta_j F_{n,m-1}(b,c)\Delta_j\|,
\]
where $C_N$ is a positive constant depending only on $N$. Hence
\begin{equation}\label{eq:traceF-1}
|\tr F_{n,m-1}(b,c)|=
\left|\tr \sum_{j=0}^n\Delta_j F_{n,m-1}(b,c)\Delta_j\right|
\le
C_N\sum_{j=0}^n\|\Delta_j F_{n,m-1}(b,c)\Delta_j\|.
\end{equation}
Taking into account that $\Delta_j P_n=P_n\Delta_j=\Delta_j$ for $j\in\{0,\dots,n\}$,
we obtain
\begin{eqnarray}\label{eq:traceF-2}
&&
\|\Delta_j F_{n,m-1}(b,c)\Delta_j\|
\le
\|\Delta_j T(c)Q_n\|\,
(\|Q_nH(b)\|\,
\|H(\widetilde{c})Q_n\|)^{m-1}
\|Q_n T(b)\Delta_j\|.
\end{eqnarray}
From Lemma~\ref{le:trunc} and the definition of the semi-norms $|\cdot|_\omega$
and $|\cdot|_\psi$ it follows that for all $n\in\Z_+$ and all $j\in\{0,\dots,n\}$,
\begin{eqnarray}
\|Q_nT(b)\Delta_j\|
&\le&
A_N\|v_-\|_\infty\left(
\max_{1\le\alpha,\beta\le N}|[u_+^{-1}]_{\alpha,\beta}|_\psi
\right)
\psi\left(\frac{1}{n-j+1}\right),
\label{eq:traceF-3}
\\
\|\Delta_jT(c)Q_n\|
&\le&
A_N\|v_+\|_\infty\left(
\max_{1\le\alpha,\beta\le N}|[u_-^{-1}]_{\alpha,\beta}|_\omega
\right)
\omega\left(\frac{1}{n-j+1}\right),
\label{eq:traceF-4}
\\
\|Q_nH(b)\|
&\le&
A_N\|v_-\|_\infty\left(
\max_{1\le\alpha,\beta\le N}|[u_+^{-1}]_{\alpha,\beta}|_\psi
\right)
\psi\left(\frac{1}{n+1}\right),
\label{eq:traceF-5}
\\
\|H(\widetilde{c})Q_n\|
&\le&
A_N\|v_+\|_\infty\left(
\max_{1\le\alpha,\beta\le N}|[u_-^{-1}]_{\alpha,\beta}|_\omega
\right)
\omega\left(\frac{1}{n+1}\right).
\label{eq:traceF-6}
\end{eqnarray}
Combining \eqref{eq:traceF-1}--\eqref{eq:traceF-6}, we get
\[
|\tr F_{n,m-1}(b,c)|
=
O\Bigg(
\left[\omega\left(\frac{1}{n+1}\right)\psi\left(\frac{1}{n+1}\right)\right]^{m-1}
\sum_{j=1}^{n+1}
\omega\left(\frac{1}{j}\right)\psi\left(\frac{1}{j}\right)
\Bigg)
\]
as $n\to\infty$. This implies that if \eqref{eq:HO1-4} holds, then
$\tr F_{n,m-1}(b,c)\to 0$ as $n\to\infty$.
\end{proof}
\subsection{Products of Hankel operators in Schatten-von Neumann classes}
\begin{lemma}\label{le:product}
Let $1\le p<\infty$, let $\omega,\psi$ belong to the Bari-Stechkin class, and let
\[
\sum_{k=1}^\infty
\left[\omega\left(\frac{1}{k}\right)\psi\left(\frac{1}{k}\right)\right]^p
<\infty.
\]

\begin{enumerate}
\item[(a)]
Suppose $a\in L_{N\times N}^\infty$ admits a factorization $a=u_-u_+$ with
$u_\pm$ satisfying \eqref{eq:factors3}. Then $H(a)H(\widetilde{a}^{-1})\in\cC_p(H_N^2)$.

\item[(b)]
Suppose $v_\pm$ and $u_\pm$ satisfy \eqref{eq:factors1} and \eqref{eq:factors3}.
Then $H(\widetilde{c})H(b)$ and $H(b)H(\widetilde{c})$ belong to $\cC_p(H_N^2)$.
\end{enumerate}
\end{lemma}
\begin{proof}
This statement is proved by analogy with \cite[Lemma~10.36]{BS06}. Let us prove
only part (a). By \cite[Proposition~2.14]{BS06}, $H(a)=T(u_-)H(u_+)$ and
$H(\widetilde{a}^{-1})=T(\widetilde{u_+^{-1}})H(u_-^{-1})$. For $n\in\Z_+$,
let $p_n(u_+)$ and $p_n(u_-^{-1})$ be the polynomials in $\cP_{N\times N}^n$
of best uniform approximation of $u_+$ and $u_-^{-1}$, respectively.
Observe that
\[
\dim\im (T(u_-)H[p_n(u_+)])\le n+1,
\quad
\dim\im (T(\widetilde{u_+^{-1}})H[p_n(u_-^{-1})])\le n+1,
\]
whence
\begin{equation}\label{eq:product-2}
s_{n+1}(H(a))
\le
\|T(u_-)H(u_+)-T(u_-)H[p_n(u_+)]\|
\le
O(\|u_+-p_n(u_+)\|_C)
\end{equation}
and similarly
\begin{equation}\label{eq:product-3}
s_{n+1}(H(\widetilde{a}^{-1}))
\le
O(
\|u_-^{-1}-p_n(u_-^{-1})\|_C).
\end{equation}
From \eqref{eq:approx1}, \eqref{eq:approx2}, and the definition of the seminorms
$|\cdot|_\omega$ and $|\cdot|_\psi$ it follows that
\begin{eqnarray}
\|u_+-p_n(u_+)\|_C
&\le&
A_N\left(\max_{1\le\alpha,\beta\le N} |[u_+]_{\alpha,\beta}|_\psi\right)
\psi\left(\frac{1}{n+1}\right),
\label{eq:product-4}
\\
\|u_-^{-1}-p_n(u_-^{-1})\|_C
&\le&
A_N\left(\max_{1\le\alpha,\beta\le N} |[u_-^{-1}]_{\alpha,\beta}|_\omega\right)
\omega\left(\frac{1}{n+1}\right),
\label{eq:product-5}
\end{eqnarray}
where $A_N$ is a positive constant depending only on $\omega,\psi,N$.
Combining \eqref{eq:product-2}--\eqref{eq:product-5}, we get
\begin{equation}\label{eq:product-6}
s_n(H(a))=O\big(\psi(1/n)\big),
\quad
s_n(H(\widetilde{a}^{-1}))=O\big(\omega(1/n)\big)
\quad
(n\in\N).
\end{equation}
From \eqref{eq:product-6} and Horn's theorem (see, e.g. \cite[Chap.~II, Theorem~4.2]{GK69})
it follows that
\[
\sum_{k=1}^\infty s_k^p\big(H(a)H(\widetilde{a}^{-1})\big)
\le
\sum_{k=1}^\infty s_k^p\big(H(a)\big) s_k^p\big(H(\widetilde{a}^{-1})\big)
=
O\left(\sum_{k=1}^\infty
\left[\omega\left(\frac{1}{k}\right)\psi\left(\frac{1}{k}\right)\right]^p\right),
\]
which finishes the proof of part (a). Part (b) is proved similarly.
\end{proof}
\section{Proofs of the main results}
\label{sec:proofs}
\subsection{Decomposition of the trace of the logarithm of one matrix series}
\begin{lemma}\label{le:decomposition}
Let $\omega,\psi$ belong to the Bari-Stechkin class and let $\Sigma$ be a compact
set in the complex plane. Suppose
\[
\begin{array}{ll}
v_-:\Sigma\to (\overline{H^\infty})_{N\times N}, &
v_+:\Sigma\to H_{N\times N}^\infty,
\\[3mm]
u_-^{\pm 1}:\Sigma\to(\cH^\omega\cap \overline{H^\infty})_{N\times N}, &
u_+^{\pm 1}:\Sigma\to(\cH^\psi\cap H^\infty)_{N\times N}
\end{array}
\]
are continuous functions. If $m\in\N$, then there exist a constant $C_m\in(0,\infty)$
and a number $n_0\in\N$ such that for all $\lambda\in\Sigma$ and all $n\ge n_0$,
\[
\tr\!\log\left\{I-\sum_{k=0}^\infty G_{n,k}\big(b(\lambda),c(\lambda)\big)\right\}
+\tr\!\left[\sum_{j=1}^{m-1}\frac{1}{j}\left(
\sum_{k=0}^{m-j-1}G_{n,k}\big(b(\lambda),c(\lambda)\big)
\right)^j\right]
=
s_n(\lambda)
\]
and
\begin{eqnarray*}
|s_n(\lambda)|
&\le&
C_m\big(\|v_-(\lambda)\|_\infty\|v_+(\lambda)\|_\infty\big)^m
\\
&&\times
\left[
\max_{1\le\alpha,\beta\le N}
\omega\left([u_-^{-1}(\lambda)]_{\alpha,\beta},\frac{1}{n+1}\right)
\right]^m
\\
&&\times
\left[
\max_{1\le\alpha,\beta\le N}
\omega\left([u_+^{-1}(\lambda)]_{\alpha,\beta},\frac{1}{n+1}\right)
\right]^m.
\end{eqnarray*}
\end{lemma}
\begin{proof}
From Lemma~\ref{le:trunc} it follows that
\begin{eqnarray*}
\big\|G_{n,k}\big(b(\lambda),c(\lambda)\big)\big\|
&\le&
\big(A_N^2\|v_-(\lambda)\|_\infty\|v_+(\lambda)\|_\infty\big)^{k+1}
\\
&&\times
\left[
\max_{1\le\alpha,\beta\le N}
\omega\left([u_-^{-1}(\lambda)]_{\alpha,\beta},\frac{1}{n+1}\right)
\right]^{k+1}
\\
&&\times
\left[
\max_{1\le\alpha,\beta\le N}
\omega\left([u_+^{-1}(\lambda)]_{\alpha,\beta},\frac{1}{n+1}\right)
\right]^{k+1}.
\end{eqnarray*}
for all $n,k\in\Z_+$ and all $\lambda\in\Sigma$. Moreover,
\[
\begin{split}
& A_N^2 \|v_-(\lambda)\|_\infty\|v_+(\lambda)\|_\infty
\\
&\quad\times
\left[
\max_{1\le\alpha,\beta\le N}
\omega\left([u_-^{-1}(\lambda)]_{\alpha,\beta},\frac{1}{n+1}\right)
\right]
\left[
\max_{1\le\alpha,\beta\le N}
\omega\left([u_+^{-1}(\lambda)]_{\alpha,\beta},\frac{1}{n+1}\right)
\right]
\\
&\le
A_N^2\max_{\lambda\in\Sigma}
\big(\|v_-(\lambda)\|_\infty\|v_+(\lambda)\|_\infty
\|u_-^{-1}(\lambda)\|_\omega
\|u_+^{-1}(\lambda)\|_\psi\big)
\omega\left(\frac{1}{n+1}\right)\psi\left(\frac{1}{n+1}\right).
\end{split}
\]
Since $\omega(1/n)\to 0$ and $\psi(1/n)\to 0$ as $n\to\infty$, there exists
a number $n_0\in\N$ such that the left-hand side of the latter inequality is
less than one for all $\lambda\in\Sigma$ and all $n\ge n_0$. Now the proof
can be developed by analogy with \cite[Proposition~3.3]{K-ZAA}.
\end{proof}
\subsection{Proof of Theorem~\ref{th:HO1}}
\begin{proof}[Proof of part {\rm (a)}]
Since $\omega$ and $\psi$ belong to the Bari-Stechkin class, there exist
$\alpha,\beta\in(0,1)$ such that $\omega(x)/x^\alpha$ and $\psi(x)/x^\beta$
are almost increasing (see \cite[Lemma~2]{BS56} or \cite[p.~54]{GM80}). Hence
there exists a constant $A>0$ such that $\omega(x)\le Ax^\gamma$ and
$\psi(x)\le Ax^\gamma$ for all $x\in(0,\pi]$, where $\gamma:=\min\{\alpha,\beta\}\in(0,1)$.
Therefore $u_-\in(\cH^\omega)_{N\times N}\subset C_{N\times N}^\gamma$ and
$u_+\in(\cH^\psi)_{N\times N}\subset C_{N\times N}^\gamma$, where $C^\gamma$
is the standard H\"older space generated by $h(x)=x^\gamma$. Since $T(\widetilde{a})$
is invertible on $H_N^2$ and $a=u_-u_+\in C_{N\times N}^\gamma$, by
Theorem~\ref{th:factorization-Hoelder}(b), the function $a$ admits a canonical
left Wiener-Hopf factorization $a=v_+v_-$ in $C_{N\times N}^\gamma$. In particular,
$v_-\in G(\overline{H^\infty})_{N\times N}$ and $v_+\in GH_{N\times N}^\infty$.
\end{proof}
\begin{proof}[Proof of parts {\rm (b)} and {\rm(c)}]
From Theorem~\ref{th:HO1}(a) it follows that hypotheses (i) and (ii)
of Theorem~\ref{th:HO-abstract} are satisfied. Suppose $m\in\N$ and
\eqref{eq:HO1-2} holds. In view of Lemma~\ref{le:product} and
\cite[Proposition~2.14]{BS06},
\[
I-T(a)T(a^{-1})=H(a)H(\widetilde{a}^{-1})\in\cC_m(H_N^2),
\quad
I-T(\widetilde{c})T(\widetilde{b})=H(\widetilde{c})H(b)\in\cC_m(H_N^2),
\]
and $H(b)H(\widetilde{c})\in\cC_m(H_N^2)$. Hence Theorems~\ref{th:HO1}(b)
and \ref{th:HO1}(c) follow from Theorems~\ref{th:HO-abstract}(a) and
\ref{th:HO-abstract}(b).
\end{proof}
\begin{proof}[Proof of part {\rm (d)}]
By Lemma~\ref{le:traceF}, $\tr F_{n,m-1}(b,c)$ as $n\to\infty$. Hence the statement
follows from the arguments of the proof of part (c) and Theorem~\ref{th:HO-abstract}(c).
\end{proof}
\begin{proof}[Proof of part {\rm (e)}]
Suppose $\Sigma$ consists of one point $\lambda$ only (and we will not
write the dependence on it). From Lema~\ref{le:decomposition} it follows that
there exist a positive constant $C_m$ and a number $n_0\in\N$ such that for
all $n\ge n_0$,
\begin{equation}\label{eq:HO1-E-1}
\tr\log\left\{I-\sum_{k=0}^\infty G_{n,k}(b,c)\right\}
=-\tr\left[\sum_{j=1}^{m-1}\frac{1}{j}\left(
\sum_{k=0}^{m-j-1}G_{n,k}(b,c)
\right)^j\right]
+s_n,
\end{equation}
where
\begin{equation}\label{eq:HO1-E-2}
|s_n|\le C_m\left[
\|v_-\|_\infty\|v_+\|_\infty\|u_-^{-1}\|_\omega\|u_+^{-1}\|_\psi\,
\omega\left(\frac{1}{n+1}\right)\psi\left(\frac{1}{n+1}\right)
\right]^m.
\end{equation}
From \eqref{eq:HO1-2} and \eqref{eq:HO1-E-2} we get $\sum_{n=n_0}^\infty |s_n|<\infty$.
Applying Lemma~\ref{le:BS} to the decomposition \eqref{eq:HO1-E-1}, we conclude
that  there exists a constant $E(a)\ne 0$ such that for all $n\ge n_0$,
\begin{equation}\label{eq:HO1-E-3}
\begin{split}
\log\det T_n(a)
=&
(n+1)\log G(a)+\log E(a)
\\
&+
\tr
\left[\sum_{\ell=1}^n\sum_{j=1}^{m-1}\frac{1}{j}\left(
\sum_{k=0}^{m-j-1}G_{\ell,k}(b,c)
\right)^j\right]
+\sum_{k=n+1}^\infty s_k.
\end{split}
\end{equation}
From \eqref{eq:HO1-E-2} we get
\begin{equation}\label{eq:HO1-E-4}
\sum_{k=n+1}^\infty s_k=O\left(\sum_{k=n+1}^\infty
\left[\omega\left(\frac{1}{k}\right)\psi\left(\frac{1}{k}\right)\right]^m\right)
\quad(n\to\infty).
\end{equation}
Combining \eqref{eq:HO1-E-3} and \eqref{eq:HO1-E-4}, we arrive at \eqref{eq:HO1-6}.
Part (e) is proved.
\end{proof}
\begin{proof}[Proof of part {\rm (f)}]
In view of \eqref{eq:HO1-E-3}, it is sufficient to show that
\begin{equation}\label{eq:HO1-F-1}
\sum_{k=n+1}^\infty s_k=o\left(\sum_{k=n+1}^\infty
\left[\omega\left(\frac{1}{k}\right)\psi\left(\frac{1}{k}\right)\right]^m\right)
\quad (n\to\infty).
\end{equation}
By Lemma~\ref{le:decomposition}, for all $k\ge n_0$,
\begin{equation}\label{eq:HO1-F-2}
\begin{split}
|s_k|\le &\
C_m\big(\|v_-\|_\infty\|v_+\|_\infty\big)^m
\\
&\times
\left[
\max_{1\le\alpha,\beta\le N}
\omega\left([u_-^{-1}]_{\alpha,\beta},\frac{1}{k+1}\right)
\right]^m
\\
&\times
\left[
\max_{1\le\alpha,\beta\le N}
\omega\left([u_+^{-1}]_{\alpha,\beta},\frac{1}{k+1}\right)
\right]^m.
\end{split}
\end{equation}
If $u_-\in G(\cH_0^\omega\cap \overline{H^\infty})_{N\times N}$, then
for every $\eps>0$ there exists a number $n_1(\eps)\ge n_0$ such that for all
$k\ge n_1(\eps)$,
\begin{equation}\label{eq:HO1-F-3}
\max_{1\le\alpha,\beta\le N}
\omega\left([u_-^{-1}]_{\alpha,\beta},\frac{1}{k+1}\right)
<\eps\omega\left(\frac{1}{k+1}\right).
\end{equation}
From \eqref{eq:HO1-F-2} and \eqref{eq:HO1-F-3} it follows that for all $n\ge n_1(\eps)$,
\[
\sum_{k=n+1}^\infty |s_k|\le\eps^mC_m
\big(
\|v_-\|_\infty
\|v_+\|_\infty
\|u_+^{-1}\|_\psi
\big)^m
\sum_{k=n+1}^\infty
\left[\omega\left(\frac{1}{k}\right)\psi\left(\frac{1}{k}\right)\right]^m,
\]
that is, \eqref{eq:HO1-F-1} holds.
If $u_+\in G(\cH_0^\psi\cap H^\infty)_{N\times N}$, then one can show
as above that \eqref{eq:HO1-F-1} is fulfilled.
\end{proof}
\subsection{Auxiliary lemma}
\begin{lemma}\label{le:uniform}
Let $\omega$ belong to the Bari-Stechkin class. If $\Sigma$ is a compact set
in the complex plane and $a:\Sigma\to(\cH_0^\omega)_{N\times N}$ is a continuous
function, then
\[
\lim_{n\to\infty}\Bigg\{
\left[\omega\left(\frac{1}{n}\right)\right]^{-1}
\sup_{\lambda\in\Sigma}\max_{1\le\alpha,\beta\le N}
\omega\left([a(\lambda)]_{\alpha,\beta},\frac{1}{n}\right)
\Bigg\}=0.
\]
\end{lemma}
\begin{proof}
Assume the contrary. Then there exist a constant $C>0$ and a sequence
$\{n_k\}_{k=1}^\infty$ such that
\[
\lim_{k\to\infty}\Bigg\{
\left[\omega\left(\frac{1}{n_k}\right)\right]^{-1}
\sup_{\lambda\in\Sigma}\max_{1\le\alpha,\beta\le N}
\omega\left([a(\lambda)]_{\alpha,\beta},\frac{1}{n_k}\right)
\Bigg\}=C.
\]
Hence there exist a number $k_0\in\N$ and a sequence $\{\lambda_k\}_{k=k_0}^\infty$
such that for all $k\ge k_0$,
\begin{equation}\label{eq:uniform-1}
\left[\omega\left(\frac{1}{n_k}\right)\right]^{-1}
\max_{1\le\alpha,\beta\le N}
\omega\left([a(\lambda_k)]_{\alpha,\beta},\frac{1}{n_k}\right)
\ge\frac{C}{2}>0.
\end{equation}
Since $\{\lambda_k\}_{k=k_0}^\infty$ is bounded, there is its convergent subsequence
$\{\lambda_{k_j}\}_{j=1}^\infty$. Let $\lambda_0$ be the limit of this subsequence.
Clearly, $\lambda_0\in\Sigma$ because $\Sigma$ is closed. Since the function
$a:\Sigma\to(\cH_0^\omega)_{N\times N}$ is continuous at $\lambda_0$, for every
$\eps\in(0,C/2)$, there exists a $\Delta>0$ such that $|\lambda-\lambda_0|<\Delta$,
$\lambda\in\Sigma$ implies $\|a(\lambda)-a(\lambda_0)\|_\omega<\eps$. Because
$\lambda_{k_j}\to\lambda_0$ as $j\to\infty$, for that $\Delta$ there exists a number
$J\in\N$ such that $|\lambda_{k_j}-\lambda_0|<\Delta$ for all $j\ge J$, and thus
\begin{equation}\label{eq:uniform-2}
\|a(\lambda_{k_j})-a(\lambda_0)\|_\omega<\eps
\quad\mbox{for all}\quad j\ge J.
\end{equation}
On the other hand, \eqref{eq:uniform-1} implies that
\begin{equation}\label{eq:uniform-3}
\left[\omega\left(\frac{1}{n_{k_j}}\right)\right]^{-1}
\max_{1\le\alpha,\beta\le N}
\omega\left([a(\lambda_{k_j})]_{\alpha,\beta},\frac{1}{n_{k_j}}\right)
\ge\frac{C}{2}>0
\quad\mbox{for all}\quad j\ge J.
\end{equation}
It is easy to see that if $f,g\in\cH^\omega$, then for all $x\in(0,\pi]$,
\[
\frac{\omega(f,x)}{\omega(x)}
\le
\frac{\omega(g,x)}{\omega(x)}
+
|f-g|_\omega.
\]
Hence, for all $j\ge J$,
\begin{equation}\label{eq:uniform-4}
\begin{split}
&
\left[\omega\left(\frac{1}{n_{k_j}}\right)\right]^{-1}
\max_{1\le\alpha,\beta\le N}
\omega\left([a(\lambda_{k_j})]_{\alpha,\beta},\frac{1}{n_{k_j}}\right)
\\
&
\le
\left[\omega\left(\frac{1}{n_{k_j}}\right)\right]^{-1}
\max_{1\le\alpha,\beta\le N}
\omega\left([a(\lambda_0)]_{\alpha,\beta},\frac{1}{n_{k_j}}\right)
+
\|a(\lambda_{k_j})-a(\lambda_0)\|_\omega.
\end{split}
\end{equation}
From \eqref{eq:uniform-2}--\eqref{eq:uniform-4} we get for all $j\ge J$,
\[
\left[\omega\left(\frac{1}{n_{k_j}}\right)\right]^{-1}
\max_{1\le\alpha,\beta\le N}
\omega\left([a(\lambda_0)]_{\alpha,\beta},\frac{1}{n_{k_j}}\right)
\ge\frac{C}{2}-\eps>0.
\]
It follows that there exist a pair $\alpha_0,\beta_0\in\{1,\dots,N\}$
and a subsequence $\{m_s\}_{s\in\N}$ of $\{n_{k_j}\}_{j=J}^\infty$ such that
for all $s\in\N$,
\[
\left[\omega\left(\frac{1}{m_s}\right)\right]^{-1}
\omega\left([a(\lambda_0)]_{\alpha_0,\beta_0},\frac{1}{m_s}\right)
\ge\frac{C}{2}-\eps>0.
\]
This contradicts the fact that $[a(\lambda_0)]_{\alpha_0,\beta_0}\in\cH_0^\omega$.
\end{proof}
\subsection{Proof of Theorem~\ref{th:HO2}}
\begin{proof}[Proof of part {\rm (a)}]
Similarly to the proof of Lemma~\ref{le:product} one can show that if
$a$ belongs to $(\cH^\omega)_{N\times N}$, then
\begin{equation}\label{eq:HO2-proof-1}
s_n\big(H(a)\big)=O\big(\omega(1/n)\big),
\quad
s_n\big(H(\widetilde{a})\big)=O\big(\omega(1/n)\big)
\quad
(n\in\N).
\end{equation}
Combining \eqref{eq:HO2-1} and \eqref{eq:HO2-proof-1}, we get
$H(a),H(\widetilde{a})\in\cC_2(H_N^2)$. It is well known that
\[
(K_{2,2}^{1/2,1/2})_{N\times N}=\big\{
a\in L_{N\times N}^\infty\ :\
H(a),H(\widetilde{a})\in\cC_2(H_N^2)
\big\}
\]
(see \cite[Section~5.1]{BS99}, \cite[Sections~10.8--10.11]{BS06}),
which finishes the proof.
\end{proof}
\begin{proof}[Proof of part {\rm (b)}]
By Theorem~\ref{th:factorization-Hoelder}, the function $a$ admits canonical
right and left Wiener-Hopf factorizations in $\cH_{N\times N}$. From
Theorem~\ref{th:HO1}(b) we get
\begin{equation}\label{eq:HO2-B-1}
\log\det T_n(a)=(n+1)\log G(a)+\log\mathrm{det}_1 T(a)T(a^{-1})+o(1)
\quad
(n\to\infty).
\end{equation}
On the other hand, from Theorem~\ref{th:HO1}(e), (f) it follows that
there exists a nonzero constant $E(a)$ such that
\begin{equation}\label{eq:HO2-B-2}
\log\det T_n(a)=(n+1)\log G(a)+\log E(a)+\delta(n,\cH)
\quad
(n\to\infty).
\end{equation}
From \eqref{eq:HO2-B-1} and \eqref{eq:HO2-B-2} we deduce that
$E(a)=\det_1 T(a)T(a^{-1})$, that is, we arrive at \eqref{eq:Widom1} with
$o(1)$ replaced by $\delta(n,\cH)$.
\end{proof}
\begin{proof}[Proof of part {\rm (c)}]
This statement is proved by analogy with \cite[Theorem~1.5]{K-ZAA} and
\cite[Theorem~1.4]{K-IWOTA05}, although the idea of this proof goes back to
\cite[Theorem~6.2]{Widom76}.
Let $\Omega$ be any bounded open set containing
the set $\mathrm{sp}\,T(a)\cup\mathrm{sp}\,T(\widetilde{a})$ on the closure of
which $f$ is analytic and let $\Sigma$ be a closed neighborhood of its boundary
$\partial\Omega$ such that
$\Sigma\cap(\mathrm{sp}\,T(a)\cup\mathrm{sp}\,T(\widetilde{a}))=\emptyset$.
Let $\lambda\in\Sigma$. Then $T(a)-\lambda I=T[a-\lambda]$ and
$T(\widetilde{a})-\lambda I=T[(a-\lambda)\widetilde{\hspace{2mm}}]$ are
invertible on $H_N^2$. By Theorem~\ref{th:factorization-Hoelder}, $a-\lambda$ admits
canonical right and left Wiener-Hopf factorizations
$a-\lambda=u_-(\lambda)u_+(\lambda)=v_+(\lambda)v_-(\lambda)$ in $\cH_{N\times N}$.
Since $a-\lambda:\Sigma\to\cH_{N\times N}$ is a continuous function with respect
to $\lambda$, in view of Theorem~\ref{th:stability}, these factorizations can be
chosen so that the functions
\[
u_-^{\pm 1},v_-^{\pm 1}:\Sigma\to(\cH\cap\overline{H^\infty})_{N\times N},
\quad
u_+^{\pm 1},v_+^{\pm 1}:\Sigma\to(\cH\cap H^\infty)_{N\times N}
\]
are continuous. From Lemma~\ref{le:decomposition} with $m=1$ it follows that there
exist a constant $C_1\in(0,\infty)$ and a number $n_0\in\N$ such that for all
$\lambda\in\Sigma$ and all $n\ge n_0$,
\begin{equation}\label{eq:HO2-C-1}
\begin{split}
|s_n(\lambda)|
\le &
C_1 \|v_-(\lambda)\|_\infty \|v_+(\lambda)\|_\infty
\left[\max_{1\le\alpha,\beta\le N}
\omega\left([u_-^{-1}(\lambda)]_{\alpha,\beta},\frac{1}{n+1}\right)\right]
\\
&\times
\left[\max_{1\le\alpha,\beta\le N}
\omega\left([u_+^{-1}(\lambda)]_{\alpha,\beta},\frac{1}{n+1}\right)\right],
\end{split}
\end{equation}
where
\[
s_n(\lambda)=
\tr\log\left\{I-\sum_{k=0}^\infty G_{n,k}\big(b(\lambda),c(\lambda)\big)\right\}.
\]
If $a\in(\cH^\omega)_{N\times N}$, then from \eqref{eq:HO2-C-1} we get
for all $\lambda\in\Sigma$ and all $n\ge n_0$,
\begin{eqnarray}\label{eq:HO2-C-2}
&&
|s_n(\lambda)|\le C_1 \max_{\lambda\in\Sigma}\big(
\|v_-(\lambda)\|_\omega
\|v_+(\lambda)\|_\omega
\|u_-^{-1}(\lambda)\|_\omega
\|u_+^{-1}(\lambda)\|_\omega
\big)
\left[\omega\left(\frac{1}{n+1}\right)\right]^2.
\end{eqnarray}
If $a\in(\cH_0^\omega)_{N\times N}$, then from Lemma~\ref{le:uniform} it follows
that for every $\eps>0$ there exists a number $n_1(\eps)\ge n_0$ such that for all
$\lambda\in\Sigma$ and all $n\ge n_1(\eps)$,
\begin{equation}\label{eq:HO2-C-3}
\max_{1\le\alpha,\beta\le N}
\omega\left([u_\pm^{-1}(\lambda)]_{\alpha,\beta},\frac{1}{n+1}\right)
<\eps\omega\left(\frac{1}{n+1}\right).
\end{equation}
Combining \eqref{eq:HO2-C-1} and \eqref{eq:HO2-C-3}, we obtain for all
$\lambda\in\Sigma$ and all $n\ge n_1(\eps)$,
\begin{equation}\label{eq:HO2-C-4}
|s_n(\lambda)|\le \eps^2 C_1 \max_{\lambda\in\Sigma}\big(
\|v_-(\lambda)\|_\omega
\|v_+(\lambda)\|_\omega
\big)
\left[\omega\left(\frac{1}{n+1}\right)\right]^2.
\end{equation}
From \eqref{eq:HO2-C-2} and \eqref{eq:HO2-C-4} we get for all $\lambda\in\Sigma$,
\begin{eqnarray}
\sum_{k=n+1}^\infty|s_k(\lambda)|
\le
\mathrm{const}\sum_{k=n+1}^\infty\left[\omega\left(\frac{1}{k}\right)\right]^2
&\mbox{if}&
a\in(\cH^\omega)_{N\times N}, \ n\ge n_0,
\label{eq:HO2-C-5}
\\
\sum_{k=n+1}^\infty|s_k(\lambda)|
\le
\eps^2\,
\mathrm{const}\sum_{k=n+1}^\infty\left[\omega\left(\frac{1}{k}\right)\right]^2
&\mbox{if}&
a\in(\cH_0^\omega)_{N\times N}, \ n\ge n_1(\eps).
\label{eq:HO2-C-6}
\end{eqnarray}
From Lemma~\ref{le:BS} and Theorem~\ref{th:HO1}(b) it follows that for all
$\lambda\in\Sigma$ and $n\ge n_0$,
\[
\log\det T_n(a-\lambda)=
(n+1)\log G(a-\lambda)+
\log\mathrm{det}_1 T[a-\lambda]T[(a-\lambda)^{-1}]+
\sum_{k=n+1}^\infty s_k(\lambda).
\]
Multiplying this equality by $-f'(\lambda)$ and then integrating over
$\partial{\Omega}$ by parts, we get
\begin{equation}\label{eq:HO2-C-7}
\begin{split}
\int_{\partial\Omega} f(\lambda)\frac{d}{d\lambda}\log\det T_n(a-\lambda)d\lambda
&=
(n+1)\int_{\partial\Omega}f(\lambda)\frac{d}{d\lambda}\log G(a-\lambda)d\lambda
\\
&
+
2\pi i E_f(a)
-
\int_{\partial\Omega}f'(\lambda)\left(\sum_{k=n+1}^\infty s_k(\lambda)\right)d\lambda.
\end{split}
\end{equation}
It was obtained in the proof of \cite[Theorem~6.2]{Widom76} (see also
\cite[Theorem~5.6]{BS99} and \cite[Section~10.90]{BS06}) that
\begin{eqnarray}
&&
\frac{1}{2\pi i}\int_{\partial\Omega}
f(\lambda)\frac{d}{d\lambda}\log\det T_n(a-\lambda)d\lambda
=\tr f(T_n(a)),
\label{eq:HO2-C-8}
\\
&&
\frac{1}{2\pi i}\int_{\partial\Omega}f(\lambda)\frac{d}{d\lambda}\log G(a-\lambda)d\lambda
=G_f(a).
\label{eq:HO2-C-9}
\end{eqnarray}
From \eqref{eq:HO2-C-5} and \eqref{eq:HO2-C-6} it follows that
\begin{equation}\label{eq:HO2-C-10}
-\int_{\partial\Omega}f'(\lambda)\left(\sum_{k=n+1}^\infty s_k(\lambda)\right)d\lambda
=\delta(n,\cH)\quad(n\to\infty).
\end{equation}
Combining \eqref{eq:HO2-C-7}--\eqref{eq:HO2-C-10}, we arrive at \eqref{eq:Widom2}
with $o(1)$ replaced by $\delta(n,\cH)$.
\end{proof}

\end{document}